\documentclass[12pt,reqno,draft]{amsart}
\usepackage{amsmath,amsthm}
\usepackage[T2A]{fontenc}
 \usepackage[cp1251]{inputenc}
 \usepackage[english,russian]{babel}
\usepackage{srctex}
\usepackage{amsfonts}
\usepackage{latexsym}

\newcommand*{\wo}{\mathcal O}
\newcommand*{\wc}{\widehat \xi}
\newcommand*{\wu}{\widehat u}
\newcommand*{\ov}{\overline}
\newcommand*{\wx}{\widehat x}
\newcommand*{\wou}{\widehat{\overline u}}
\newcommand*{\wF}{\widehat F}

\newtheorem{theorem}{Theorem}
\newtheorem{corollary}{Corollary}
\newtheorem{proposition}{Proposition}
\newtheorem{definition}{Definition}

\begin{document}


\title[Controllability and second-order optimality conditions]{Controllability and necessary second-order
optimality conditions
in optimal control problems}
\author{E.\,R.~Avakov, G.\,G.~Magaril-Il'yaev}

\thanks{This research was carried out with the financial support of the Russian Foundation for Basic
Research (grant no.\ 17-01-00649).}

\address{Institute of Control Sciences\\ of the Russian Academy of Sciences;\\ Moscow State University}

\maketitle

\begin{center}
{\sc Abstract}
\end{center}
\medskip

The paper puts forward sufficient conditions for local
controllability of a~control dynamical system. The results obtained
are meaningful in the case when the linear approximation to this
system is not completely controllable. As a~corollary, we obtain
necessary second-order optimality conditions for a~general optimal
control problem.

Bibliography: 12 titles.


\section*{Introduction}
The main result of the present paper gives sufficient conditions for local
controllability of an abstract control system. As a~direct corollary of this result we
obtain second-order optimality conditions for an abstract variant of an optimal control
problem. The general results obtained below are applied to a~control dynamical system,
which establishes sufficient conditions for its local controllability which are
meaningful in the case when the linear approximation to this system is not completely
controllable. From these conditions we readily obtain necessary second-order optimality
conditions for a strong minimum in an optimal control problem.  Consideration of an
abstract control system, which in our opinion has an independent interest, enables one to
give a~complete investigation of the questions of interest, without the distraction of
special properties of dynamical systems described by ordinary differential equations.

The paper has three sections. In the first section we consider an abstract control
system, prove the main result on the conditions of its local controllability, and
establish a~corollary on necessary second-order optimality conditions for an abstract
optimal control problem. Note that an important tool for the proof of the main result is
the  special inverse theorem, which has independent interest. The second section is
concerned with applications of these results to a~control dynamical system of a~fairly
general form. At the end of the second section some comments are given. In the third
section we apply the theorem on controllability of a~dynamical system to the so-called
systems of first order of abnormality. We shall also consider some examples showing, in
particular, that the conditions guaranteeing the local controllability of the system are
substantial.

\section{Abstract control system}
Let $X$, $Y$ and $Z$ be normed linear spaces, $\mathcal U\subset Z$, $F\colon
\mathbb R^n\times X\times\mathcal U\to Y$, $f\colon \mathbb R^n\times X\to\mathbb
R^{m_1}$ and $g\colon \mathbb R^n\times X\to\mathbb R^{m_2}$. Consider the control system
\begin{equation}\label{1}
F(\xi,x,u)=0,\quad u\in \mathcal U,\quad f(\xi,x)\le0,\quad g(\xi,x)=0,
\end{equation}
where the inequality is understood coordinatewise.

System \eqref{1} models a~control dynamical system encountered in
optimal control problems: $x$~is the phase variable, $u$~is the control,
the variable~$\xi$ allows one to take into account fairly general boundary conditions.

A point $(\widehat \xi ,\widehat x ,\widehat u )\in\mathbb R^n\times X\times\mathcal U$ will be called {\it admissible} for the
control system \eqref{1} if it satisfies all the relations in~\eqref{1} and
$\widehat u \in\operatorname{int}\mathcal U$.

\begin{definition}\label{d1}
A~control system \eqref{1} will be called \textit{locally controllable} with respect to
an admissible point $(\widehat \xi ,\widehat x ,\widehat u )$ if, for any
neighbourhood~$W$ of the point $(\widehat \xi ,\widehat x )$, there exist neighbourhoods
$W_1$ and $W_2$ of the  $\mathbb R^{m_1}$- and $\mathbb R^{m_2}$-origins, respectively,
such that for any $y=(y_1,y_2)\in W_1\times W_2$ there exists an element
$(\xi_y,x_y,u_y)\in W\times \mathcal U$ for which $F(\xi_y,x_y,u_y)=0$, \
$f(\xi_y,x_y)\le y_1$ and  \ $g(\xi_y,x_y)=y_2$.
\end{definition}

We introduce some notation. Let $X$ and $Y$ be normed linear spaces, $X^*$ and $Y^*$ be
their dual spaces. Given  a~continuous linear operator $A\colon X\to Y$, $A^*$ denotes
the adjoint operator of~$A$. We let $\langle x^*,x\rangle$ denote a~linear functional
$x^*\in X^*$ evaluated at an element $x\in X$. The dual $(\mathbb R^n)^*$ of~$\mathbb
R^n$ will be identified with the space of vector rows; $(\mathbb R^n)_+^*$~is the cone of
positive functionals on~$\mathbb R^n$ (that is, nonnegative vector rows).

Given a~bilinear mapping $B\colon X\times X\to Y$, we shall write $B[x_1,x_2]$
to evaluate the mapping~$B$ at an element  $(x_1,x_2)$.

If  $(\widehat \xi ,\widehat x ,\widehat u )$ is an admissible point for system \eqref{1}, then
for the derivatives\footnote{Throughout, by derivatives we shall mean the Fr\'echet derivatives.}
at this point we shall frequently use for brevity the notation: $\widehat
F'=F'(\widehat \xi ,\widehat x ,\widehat u )$, $\widehat f'=f'(\widehat \xi ,\widehat x )$, $\widehat g'=g'(\widehat \xi ,\widehat x )$, and similarly, for the
partial derivatives $\widehat F_x=F_x(\widehat \xi ,\widehat x ,\widehat u )$, $\widehat
f_\xi=f_\xi(\widehat \xi ,\widehat x )$, and so~on.

For the second derivatives of mappings $F$, $f$ and $g$ (which are identified with the corresponding
continuous symmetric bilinear forms), we shall write
$\widehat F''=F''(\widehat \xi ,\widehat x ,\widehat u )$, $\widehat f''=f''(\widehat \xi ,\widehat x )$, $\widehat g''=g''(\widehat \xi ,\widehat x )$.

Let $(\widehat \xi ,\widehat x ,\widehat u )$ be an admissible point for system \eqref{1}. Given any
$q=(\zeta,h,v)\in  \mathbb R^n\times X\times Z$ (assuming that the corresponding derivatives
exist) we consider the system of equations with respect to the variables $y^*\in
Y^*$, $\lambda_1\in(\mathbb R^{m_1})^*_+$ and $\lambda_2\in(\mathbb R^{m_2})^*$:
\begin{equation}\label{reg}
\begin{cases}
\widehat F_\xi^*y^*+ \widehat f_\xi^*\lambda_1+\widehat g_\xi^*\lambda_2=0,\\[7pt]
\widehat F_x^*y^*+ \widehat f_x^*\lambda_1+\widehat g_x^*\lambda_2=0,\\[6pt]
\min\limits_{u \in \mathcal U}\langle y^*, F(\widehat \xi ,\widehat x ,u)\rangle=\langle y^*,F(\widehat \xi ,\widehat x ,\widehat u )\rangle=0,\\[10pt]
\langle\lambda_1,f(\widehat \xi ,\widehat x )\rangle=0,\\[6pt]
\langle y^*,
\widehat F''[q,q]\rangle+\langle\lambda_1,\widehat f''[(\zeta,h),(\zeta,h)]\rangle\\[6pt]
+\langle\lambda_2,\widehat g''[(\zeta,h),(\zeta,h)]\rangle\ge0.\\[6pt]
\end{cases}
\end{equation}

We let $\Lambda(\widehat \xi ,\widehat x ,\widehat u ,q)$ denote the set of triples $(y^*,\lambda_1,\lambda_2)\in
Y^*\times(\mathbb R^{m_1})^*_+\times(\mathbb R^{m_2})^*$ satisfying all the relations in~\eqref{reg} and such that
$|\lambda_1|+|\lambda_2|\ne0$.

As was already mentioned, system \eqref{1} is an abstract model of a~control
dynamical system in an optimal control problem.
The assumptions that follow can be looked upon as abstract
variants of the assumptions and properties that hold in a~standard
optimal control problem (for more detail, see the next section).

{\bf Basic Assumptions}:
\begin{itemize}
\item[1)] $X$, $Y$ and $Z$ are Banach spaces.
\item[2)] If $(\widehat \xi ,\widehat x ,\widehat u )$ is an admissible point for system \eqref{1}, then there exists
a~neighbourhood $(\widehat \xi ,\widehat x ,\widehat u )$ in which the mapping~$F$ has a~continuous second derivative
and the
mappings $f$ and $g$ have continuous second derivatives in the projection of this
neighbourhood onto $\mathbb R^n\times X$. The operator $F_x(\widehat \xi ,\widehat x ,\widehat u )$ is invertible.
\item[3)] For any $k\in\mathbb N$, $\varepsilon>0$,
$\overline\alpha=(\alpha_1,\ldots,\alpha_k)^T\in\Sigma^k=\{\,\overline\alpha=(\alpha_1,\ldots,\alpha_k)^T\in\mathbb
R^k_+ : \sum_{i=1}^k\alpha_i<1\,\}$ and  $\overline u=(u_0,u_1,\ldots,u_k)\in\mathcal U^{k+1}$
there exists an element $M_\varepsilon(\overline\alpha,\overline u)\in\mathcal U$ such that the mapping
$\overline\alpha\mapsto M_\varepsilon(\overline\alpha,\overline u)$ is continuous on~$\Sigma^k$, and if
$(\widehat \xi ,\widehat x ,\widehat{\overline u})\in \mathbb R^n\times X\times\mathcal U^{k+1}$, then there exists
a~neighbourhood~$U$ of the point $(\widehat \xi ,\widehat x ,\widehat{\overline u})$ such that $$
\|F(\xi,x,M_\varepsilon(\overline\alpha,\overline
u))-\sum_{i=0}^k\alpha_iF(\xi,x,u_i)\|_Y<\varepsilon
$$
and
$$
\|F_x(\xi,x,M_\varepsilon(\overline\alpha,\overline
u))-\sum_{i=0}^k\alpha_iF_x(\xi,x,u_i)\|<\varepsilon
$$
for all $(\xi,x,\overline u)\in U$ and $\overline\alpha\in\Sigma^k$.
\end{itemize}

Condition $3)$ means in particular that the closures of the ranges of the mappings
$u\mapsto F(\xi,x,u)$ and $u\mapsto F_x(\xi,x,u)$ are convex sets. This condition always
holds in an optimal control problem, where $F$~is an integral operator corresponding to
the differential constraint. The quantity $M_\varepsilon(\overline\alpha,\overline u)$
will be called the {\it mix of the controls} $u_0,u_1,\ldots,u_k$. This concept was first
introduced by Tikhomirov~\cite{T} (see also \cite{AMT}). Using mix, one can reach in
a~`regular' way any point lying in the closure of the range of any of the above mappings.

Let $(\widehat \xi ,\widehat x ,\widehat u )$ be an admissible point for system \eqref{1}. We define the set
(assuming that the corresponding derivatives exist)
\begin{multline}\label{k}
K(\widehat \xi ,\widehat x ,\widehat u )=\{\,q=(\zeta,h,v)\in\mathbb R^n\times  X\times Z : \widehat F'q=0,\quad
\widehat f'[\zeta,h]\le0,\\ \widehat g'[\zeta,h]=0\,\},
\end{multline}
where $\widehat f'[\zeta,h]$ and $\widehat g'[\zeta,h]$ are linear operators
$\widehat f'$ and $\widehat g'$ evaluated at an element $(\zeta,h)$.

The main result of the present paper is the following

\begin{theorem}\label{T1}
Given an admissible point $(\widehat \xi ,\widehat x ,\widehat u )$   for the control system \eqref{1},
assume that the Basic Assumptions are satisfied and there exists
$q=(\zeta,h,v)\in K(\widehat \xi ,\widehat x ,\widehat u )$ such that $\Lambda(\widehat \xi ,\widehat x ,\widehat u ,q)=\emptyset$. Then
system \eqref{1} is locally controllable with respect to the point $(\widehat \xi ,\widehat x ,\widehat u )$.

Moreover, there exists a~constant $\kappa_0>0$ such that
$\|x_y-\widehat x \|_X+|\xi_y-\widehat \xi \,|\le \kappa_0|y|^{1/2}$
for the variables $y$, $x_y$ and
$\xi_y$ from the definition of the controllability of system~\eqref{1}.
\end{theorem}

Before proving the theorem we will prove two propositions and special inverse theorem
which guarantees the existence of inverse function with weaker assumptions than in the
classical situation. We first need some definitions.

For any $k\in\mathbb N$ and any tuple $\overline u=(u_1,\ldots u_k)\in \mathcal U^k$
we consider the mapping $\mathcal F\colon\mathbb R^n\times X\times\mathbb
R^k\times\mathcal U\to Y$ defined by
\begin{equation}\label{2f}
\mathcal F(\xi,x,\overline\alpha,u;\overline
u)=F(\xi,x,u)+\sum_{i=1}^k\alpha_i(F(\xi,x,u_i)-F(\xi,x,u)),
\end{equation}
where $\overline \alpha=(\alpha_1,\ldots,\alpha_k)^T$.

Let $(\widehat \xi ,\widehat x ,\widehat u )$ be an admissible point for system \eqref{1} and let the Basic Assumptions
hold. We have $\mathcal F(\widehat \xi ,\widehat x ,0,\widehat u ;\overline u)=F(\widehat x ,\widehat \xi ,\widehat u )=0$ and $\mathcal
F_{x}(\widehat \xi ,\widehat x ,0,\widehat u ;\overline u)=\widehat F_x$, and hence by the classical implicit function theorem,
there exists a~twice continuously differentiable mapping $(\xi,\overline\alpha,u)\mapsto
x(\xi,\overline\alpha,u;\overline u)$ from some neighbourhood of the point $(\widehat \xi ,0,\widehat u )$ such that
$\mathcal F(\xi,x(\xi,\overline\alpha,u;\overline u),\overline\alpha,u;\overline u)=0$ for all
$(\xi,\overline\alpha,u)$ from this neighbourhood.

Hence for all  such triples $(\xi,\overline\alpha,u)$ and all $r\in \mathbb R^{m_1}$
we have the continuously differentiable $\mathbb
R^{m_1+m_2}$-valued mapping~$\Phi$ defined by
\begin{equation}\label{31}
\Phi(\xi,\overline\alpha, r, u;\overline u)=(f(\xi,x(\xi,\overline\alpha, u;\overline u))+r, \
g(\xi,x(\xi,\overline\alpha, u;\overline u)))^T.
\end{equation}

We let $\Phi_{(\xi,\overline\alpha,r)}(\widehat \xi ,0,0,\widehat u ;\overline u)$ denote the partial derivative with respect to
$(\xi,\overline\alpha,r)$ of the mapping $(\xi,\overline\alpha,r,u)\mapsto \Phi(\xi,\overline\alpha,r,u;\overline
u)$ at a~point $(\widehat \xi ,0,0,\widehat u )$. We also denote by $\Phi_{ww}(\widehat \xi ,0,0,\widehat u ;\overline u)$, where $w=(\xi,u)$, the
second partial derivative with respect to~$w$ of the same mapping at the same point.

Given an element $a$ of a~linear space $X$, we let ${\rm conv}\,a$ denote the ray spanned by~$a$;
that is,  ${\rm conv}\,a=\{\,\beta a\in X : \beta\ge0\,\}$.

\begin{proposition}\label{P1}
Under the hypotheses of Theorem~\ref{T1} there exist $k\in\mathbb N$ and a~tuple
$\widehat{\overline u}=(\widehat u _1,\ldots,\widehat u _k)\in\mathcal U^{k}$ such that
\begin{multline}\label{re}
0\in\operatorname{int}\{\,\Phi_{(\xi,\overline\alpha,r)}(\widehat \xi ,0,0,\widehat u ;\widehat{\overline u})(\mathbb R^n\times\mathbb
R^k_+\times(\mathbb R^{m_1}_++f(\widehat \xi ,\widehat x )))\\+{\rm
conv}\,\Phi_{ww}(\widehat \xi ,0,0,\widehat u ;\widehat{\overline u})[(\zeta,v),(\zeta,v)]\,\}.
\end{multline}
\end{proposition}

\begin{proof}[Proof]
Assume on the contrary that the inclusion \eqref{re} does not hold
for any $k\in\mathbb N$ and any tuple $\overline
u=(u_1,\ldots,u_k)\in\mathcal U^{k}$ .
Then by the separation theorem there exists a~nonzero vector $\overline\lambda(\overline u)\in(\mathbb
R^{m_1+m_2})^*$ such that
\begin{multline}\label{2}
\langle \overline\lambda(\overline u), \Phi_{(\xi,\overline\alpha,r)}(\widehat \xi ,0,0,\widehat u ; \overline
u)[\xi,\overline\alpha,r]\\+\beta\,\Phi_{ww}(\widehat \xi ,0,0,\widehat u ; \overline u)[(z,v),(z,v)]\rangle\ge0
\end{multline}
for all $\xi\in\mathbb R^n$, $\overline\alpha\in\mathbb R^k_+$, $r\in\mathbb
R^{m_1}_++f(\widehat \xi ,\widehat x )$ and $\beta\ge0$.

By the implicit function there exist the partial derivatives $\widehat x _{\xi}(\overline u)$ and
$\widehat x _{\alpha_i}(\overline u)$ with respect to~$\xi$ and $\alpha_i$, $1\le i\le k$,
respectively, of the mapping
$(\xi,\overline\alpha,u)\mapsto x(\xi,\overline\alpha,u;\overline u)$ at the point $(\widehat \xi ,0,\widehat u )$ satisfying the
relations
\begin{equation}\label{g8}
\widehat F_x\widehat x _{\alpha_i}(\overline u)\alpha_i+\alpha_iF(\widehat \xi ,\widehat x ,u_i)=0,\quad i=1,\ldots,k,
\end{equation}
for all $\alpha_i\in\mathbb R$ and
\begin{equation}\label{g9}
\widehat F_x\widehat x _{\xi}(\overline u)\xi+\widehat F_{\xi}\xi=0
\end{equation}
for all $\xi\in\mathbb R^n$.

This implies that $\widehat x _{\alpha_i}(\overline u)$ depends only on the $i$th component of the vector
$\overline u$, while $\widehat x _{\xi}(\overline u)$ is independent of~$\overline u$. Hence in what follows we shall
write $\widehat x _{\alpha_i}(u_i)$ and $\widehat x _{\xi}$ in place of $\widehat x _{\alpha_i}(\overline u)$ and $\widehat x _{\xi}(\overline u)$, respectively.

Writing the vector $\overline \lambda(\overline u)$ in the form $\overline\lambda(\overline u)=(\lambda_1(\overline
u),\lambda_2(\overline u))$, where $\lambda_i(\overline u)\in(\mathbb R^{m_i})^*$, $i=1,2$, inequality
\eqref{2} assumes the form
\begin{multline}\label{3}
\langle \lambda_1(\overline u),\widehat f_x \widehat x _\xi\xi+\widehat f_\xi \xi+\widehat
f_x\sum_{i=1}^k\widehat x _{\alpha_i}(u_i)\alpha_i +r+f(\widehat \xi ,\widehat x )\rangle\\+\langle\lambda_2(\overline
u),\widehat g_x \widehat x _\xi\xi+\widehat g_\xi \xi+\widehat
g_x\sum_{i=1}^k\widehat x _{\alpha_i}(u_i)\alpha_i\rangle\\+\langle\overline\lambda(\overline
u),\beta\,\Phi_{ww}(\widehat \xi ,0,0,\widehat u ;\overline u)[(\zeta,v), (\zeta,v)]\rangle\ge0
\end{multline}
for all $\xi\in\mathbb R^n$, $\alpha_i\ge0$, $i=1,\ldots,k$, $r\in\mathbb R^{m_1}_+$ and
$\beta\ge0$ (the chain rule for differentiation being useful).

We shall assume that  $|\overline\lambda(\overline u)|=1$. We let
$\mathcal A(\overline u)$
denote the set of all such $\overline\lambda(\overline u)$ satisfying~\eqref{3}. It is clearly seen that
$\mathcal A(\overline u)$ is a~closed subset of the unit sphere in $(\mathbb R^{m_1+m_2})^*$.
Thus with each $k\in\mathbb N$ and each tuple $\overline u=(u_1,\ldots,u_k)$
one may associate a~closed subset of this compact set. We claim that the family $\mathcal
A$ of all such subsets has the finite intersection property.

Let $\overline u_1,\ldots,\overline u_s$  be an arbitrary finite family of tuples $\overline
u_i=(u_{i1},\ldots, u_{ik_i})$, $i=1,\ldots,s$. We claim that
$\cap_{i=1}^s\mathcal A(\overline u_i)\ne\emptyset$. Indeed, let $\widetilde{\overline u}$~be a~tuple
consisting of the union of all such families. The tuple~$\widetilde{\overline u}$ satisfies the inequality
similar to~\eqref{3} with
$\overline\lambda(\widetilde{\overline u})$ and with $k$~replaced by the cardinality of the tuple~$\widetilde{\overline u}$.
Let $1\le j\le s$. Setting in this analogue of inequality \eqref{3}\enskip $\alpha_i=0$
for such indexes~$i$ for which $u_i$ does not lie in the tuple $\overline u_j$, we see that
$\overline\lambda(\widetilde{\overline u})\in\mathcal A(\overline u_j)$ and hence $\overline\lambda(\widetilde{\overline u})
\in\cap_{j=1}^s\mathcal A(\overline u_j)$.

So, the family $\mathcal A$ of closed subsets
of the compact set has the finite intersection property and hence
there exists a~vector $\overline\lambda=(\lambda_1,\lambda_2)$,
$|\overline\lambda|=1$ for which~\eqref{3} holds for any tuple~$\overline
u$. In particular, \eqref{3}~holds for singletons  $\overline u=u_1$.
We shall write $u$~in place of~$u_1$ and since in this case $\overline\alpha=\alpha_1$,
we write~$\alpha$ in place of~$\alpha_1$.

Thus by \eqref{3} all such tuples satisfy the relation
\begin{multline}\label{4}
\langle \lambda_1,\widehat f_x \widehat x _\xi\xi+\widehat f_\xi \xi+\widehat
f_x\widehat x _{\alpha}(u)\alpha +r+f(\widehat \xi ,\widehat x )\rangle+\langle\lambda_2,\widehat g_x \widehat x _\xi\xi+\widehat
g_\xi \xi\\+\widehat
g_x\widehat x _{\alpha}(u)\alpha\rangle+\langle\overline\lambda,\beta\,\Phi_{ww}(\widehat \xi ,0,0,\widehat u ;u)[(\zeta,v),
(\zeta,v)]\rangle\ge0
\end{multline}
for all $u\in\mathcal U$, $\xi\in\mathbb R^n$, $\alpha\ge0$, $r\in\mathbb R^{m_1}_+$ and
$\beta\ge0$.

Setting in \eqref{4}\enskip $\xi=0$, $\alpha=\beta=0$ and $r=\widetilde r-f(\widehat \xi ,\widehat x )$, where
$\widetilde r\in\mathbb R^{m_1}_+$, we see that
$\langle \lambda_1,\widetilde r\rangle\ge0$ for all $\widetilde r\in\mathbb R^{m_1}_+$, and hence
$\lambda_1\in (\mathbb R^{m_1})^*_+$.

Assume that $\xi=0$, $\alpha=\beta=0$ and $r=0$ in \eqref{4}, hence  $\langle
\lambda_1,f(\widehat \xi ,\widehat x )\rangle\ge0$. But since $\lambda_1\ge0$ and $f(\widehat \xi ,\widehat x )\le0$, we have  $\langle
\lambda_1,f(\widehat \xi ,\widehat x )\rangle\le0$ and therefore
\begin{equation}\label{6}
\langle \lambda_1,f(\widehat \xi ,\widehat x )\rangle=0.
\end{equation}
We set $y^*=-(\widehat F_x^{-1})^*(\widehat f_x^*\lambda_1+\widehat g_x^*\lambda_2)$. Then
\begin{equation}\label{7} \widehat F_x^* y^*+\widehat f_x^*\lambda_1+\widehat
g_x^*\lambda_2=0.
\end{equation}

If in \eqref{4} $\xi=0$, $\beta=0$ and  $r=-f(\widehat x ,\widehat \xi \,)$, then
\begin{equation}\label{8}
\langle \lambda_1, \widehat f_x\widehat x _{\alpha}(u)\alpha\rangle+\langle \lambda_2,\widehat
g_x\widehat x _{\alpha}(u)\alpha\rangle\ge0
\end{equation}
for all  $u\in\mathcal U$ and $\alpha\ge0$.

Applying \eqref{7} to  $\widehat x _{\alpha}(u)\alpha$, we get
\begin{equation*}
\langle y^*, \widehat F_x \widehat x _{\alpha}(u)\alpha\rangle+\langle \lambda_1, \widehat
f_x\widehat x _{\alpha}(u)\alpha\rangle+\langle \lambda_2, \widehat g_x\widehat x _{\alpha}(u)\alpha\rangle=0.
\end{equation*}
Hence, using \eqref{g8} with $\alpha_i=\alpha$, $\overline u=u_i=u$ and~\eqref{8} we have,
for all $u\in\mathcal U$ and $\alpha\ge0$,
 \begin{equation*}
-\langle y^*, \widehat F_x\widehat x _{\alpha}(u)\alpha\rangle=\langle y^*,\alpha F(\widehat \xi ,\widehat x ,u)\rangle\ge0=\langle
y^*,F(\widehat \xi ,\widehat x ,\widehat u )\rangle
\end{equation*}
and therefore,
\begin{equation}\label{9}
\min_{u \in \mathcal U}\langle y^*, F(\widehat \xi ,\widehat x ,u)\rangle=\langle y^*,F(\widehat \xi ,\widehat x ,\widehat u )\rangle.
\end{equation}

Now if $\alpha=\beta=0$ and $r=-f(\widehat \xi ,\widehat x )$ in~\eqref{4}, then, since
$\xi$~is arbitrary,
\begin{equation*}
\langle\lambda_1, \widehat f_x \widehat x _\xi\xi+\widehat {f_\xi}\xi\rangle+\langle\lambda_2,\widehat g_x
\widehat x _\xi\xi+\widehat g_\xi\xi\rangle=0.
\end{equation*}
Hence, using \eqref{7}, as applied to $\widehat x _\xi\xi$, we get
\begin{equation*}
-\langle y^*,\widehat F_x\widehat x _\xi\xi\rangle+\langle\lambda_1,\widehat f_\xi\xi\rangle+\langle\lambda_2,\widehat
g_\xi\xi\rangle=0.
\end{equation*}
In combination with \eqref{g9} this means that
\begin{equation*}
\langle y^*,\widehat F_\xi \xi\rangle+\langle\lambda_1,\widehat f_\xi\xi\rangle+\langle\lambda_2,\widehat
g_\xi\xi\rangle=0,
\end{equation*}
and so
\begin{equation}\label{g13a}
\widehat F_\xi^*y^*+\widehat f_\xi^*\lambda_1+\widehat g_\xi^*\lambda_2=0.
\end{equation}

From \eqref{g13a}, \eqref{7}, \eqref{9} and \eqref{6} it follows that the triple
$(y^*,\lambda_1,\lambda_2)\in Y^*\times(\mathbb R^{m_1})^*_+\times(\mathbb R^{m_2})^*$ for
which  $|\lambda_1|+|\lambda_2|\ne0$ satisfies the first four relations in~\eqref{reg}.
We claim that it also satisfies the fifth relation in~\eqref{reg}. To this aim we shall
transform the second term in~\eqref{4}. However,
we first need a~few remarks.

We recall that  $w=(\xi,u)$. As in the above, we briefly denote the derivatives of~$F$ at a~point
$(\widehat \xi ,\widehat x ,\widehat u )$ by $\widehat F_{w}=F_{w}(\widehat \xi ,\widehat x ,\widehat u )$,
$\widehat F_{xx}=F_{xx}(\widehat \xi ,\widehat x ,\widehat u )$, $\widehat F_{xw}=F_{xw}(\widehat \xi ,\widehat x ,\widehat u )$ and so forth.

We denote by $\widehat x _{w}$
the partial derivative with respect to~$w$ of the mapping
$(\xi,\overline\alpha,u)\mapsto x(\xi,\overline\alpha,u;\overline u)$ at a~point $(\widehat \xi ,0,\widehat u )$.
By the rule for differentiation of
implicit functions, $\widehat x _{w}=-\mathcal F_x(\widehat \xi ,\widehat x ,0,\widehat u ;\overline u)\mathcal F_w(\widehat \xi ,\widehat x ,0,\widehat u ;\overline
u)=-\widehat F^{-1}_x\widehat F_w$.

By the hypothesis $(\zeta,h,v)\in K(\widehat \xi ,\widehat x ,\widehat u )$, and hence $\widehat F_xh+\widehat F_wp=0$, where
$p=(\zeta,v)$. Hence $h=-\widehat F^{-1}_x\widehat F_wp=\widehat x _wp$. Using this fact and
the well-known formula for the second derivative of an implicit function (see, for example,~\cite{Z}),
we have
\begin{multline*}
\widehat x _{ww}[p,p]=\widehat F_x^{-1}(((\widehat F_{xw}+\widehat F_{xx}\widehat x _{w})p)\widehat F_x^{-1}\widehat F_{w}p\\-
((\widehat F_{ww}+\widehat F_{wx}\widehat x _{w})p)p)=
\widehat F_x^{-1}(\widehat F_{xw}[p,\widehat F_x^{-1}\widehat F_{w}p]+\widehat F_{xx}[\widehat x _{w}p,\widehat F_x^{-1}\widehat F_{w}p]\\-
\widehat F_{ww}[p,p]-\widehat F_{wx}[\widehat x _{w}p,p])=-
\widehat F_x^{-1}(\widehat F_{xx}[h,h]+2\widehat F_{xw}[h,p]\\+\widehat F_{ww}[p,p])=-\widehat F_x^{-1}\widehat F''[q,q].
\end{multline*}

Further, direct (but routine) calculations show that
\begin{multline*}
\Phi_{ww}(\widehat \xi ,0,0,\widehat u ;\overline u)[p,p]=(\widehat f''[(\zeta,h),(\zeta,h)]+\widehat f_x\widehat x _{ww}[p,p],\\
\widehat g''[(\zeta,h),(\zeta,h)]+\widehat g_x\widehat x _{ww}[p,p]).
\end{multline*}
Substituting  here the above expression for $\widehat x _{ww}[p,p]$, it follows from \eqref{4}
with $\xi=0$, $\alpha=0$, $r=-f(\widehat \xi ,\widehat x )$ and $\beta=1$ that
\begin{multline*}
\langle\overline\lambda,\Phi_{ww}(\widehat \xi ,0,0,\widehat u ;\overline u)[p,p]\rangle=\langle\lambda_1,\widehat
f''[(\zeta,h),(\zeta,h)]\rangle\\+ \langle\lambda_2,\widehat g''[(\zeta,h),(\zeta,h)]\rangle -\langle
\widehat f_x^*\lambda_1+\widehat g_x^*\lambda_2,\widehat F_x^{-1}\widehat F''[q,q]\rangle\ge0.
\end{multline*}
Hence and from the definition of the functional $y^*$ it follows that the triple  $(y^*,\lambda_1,\lambda_2)$
also satisfies the fifth relation in~\eqref{reg}; that is,
$\Lambda(\widehat \xi ,\widehat x ,\widehat u ,q)\ne\emptyset$, contradicting the assumption.
\end{proof}

Recall that the mapping $\mathcal F$ is defined by \eqref{2f} and
$x(\xi,\overline\alpha,u;\overline u)$ is the solution of the equation $\mathcal
F(\xi,x,\overline\alpha,u;\overline u)=0$.

\begin{proposition}\label{P11} Let the assumptions of Theorem~\ref{T1} hold
and the tuple $\widehat{\overline u}=(\widehat u _1,\ldots,\widehat u _k)$ be from
Proposition \ref{P1}. There exist neighbourhoods $\wo_0(\wc\,)$, $\wo_0(0)$, $\wo_0(\wu)$
of the points $\wc$, $0\in\mathbb R^k$, $\wu$  and $\varepsilon_0>0$ such that, for all
$0<\varepsilon\le \varepsilon_0$, there exists a continuous mapping
$(\xi,\ov\alpha,u)\mapsto x_\varepsilon(\xi,\ov\alpha, u)$  from
$\wo_0(\wc\,)\times(\wo_0(0)\cap \mathbb R^k_+)\times\wo_0(\wu)$ into $\wo(\wx)$, for
which $F(\xi,x_\varepsilon(\xi,\ov\alpha, u),M_\varepsilon(\ov\alpha, (u,\wou)))=0$ and
\begin{equation}\label{t4}
\|x_\varepsilon(\xi,\ov\alpha,u)-x(\xi,\ov\alpha,u;\wou)\|_X<2\|\wF_x^{-1}\|\varepsilon
\end{equation}
for all $(\xi,\ov\alpha, u)\in \wo_0(\wc\,)\times(\wo_0(0)\cap \mathbb
R^k_+)\times\wo_0(\wu)$.
\end{proposition}
This  Proposition we do not prove since it is a~particular case of  more general
assertion proved in \cite{AM} (see Corollary 3).

Before the formulation of the inverse theorem, we introduce some definition.

Let $V$ be an open subset of a normed linear space. We let $C(V,\mathbb R^m)$ denote the
space of all bounded continuous mappings $G$ from $V$ into~$\mathbb R^m$ with the norm
$\|G\|=\sup_{w\in V}|G(w)|$.

\begin{theorem}\label{T3}
Let $X$ be a normed space, $K$ be a~convex cone in~$X$, $V$ be a~neighbourhood of a~point
$\widehat w\in K$, a~mapping $\widehat G\colon V\to\mathbb R^m$ is continuous and bounded
on~$V$ and is twice differentiable at $\widehat w$, $q\in{\rm Ker}\,\widehat G'(\widehat
w)\cap K$, $\|q\|=1$ and
\begin{equation}\label{l}
0\in{\rm int}\{\,\widehat G'(\widehat w)(K-\widehat w)+{\rm conv}\,\widehat G''(\widehat
w)[q,q]\,\}.
\end{equation}

Then there exist a~neighbourhood $V_1$ of the point $\widehat G(\widehat w)$ and
a~constant $\kappa>0$ such that, for any $y\in V_1$, there exists a~neighbourhood $V_y$
of the mapping $\widehat G\in C(V,\mathbb R^m)$ with the property that, for any $G\in
V_y$, there exists a~point $w_G(y)\in V\cap K$ for which
\begin{equation}\label{l1}
G(w_G(y))=y,\qquad \|w_G(y)-\widehat w\|_X\le \kappa|y-\widehat G(\widehat w)|^{1/2}.
\end{equation}
\end{theorem}

\begin{proof}[Proof]
Consider the linear mapping $\Lambda\colon X\times\mathbb R\to\mathbb R^m$, defined by
the formula
$$
\Lambda(w,\beta)=\widehat G'(\widehat w)w+\frac12\beta \widehat G''(\widehat w)[q,q].
$$
From condition \eqref{l} it follows that $0\in\operatorname{int}\Lambda((K-\widehat
w)\times\mathbb R_+)$. In turn, this implies that there exist $\rho>0$ for which
$U_{\mathbb R^m}(0,\rho)\subset \Lambda((K-\widehat w)\times\mathbb R_+)$, a~continuous
mapping $R=(R_1,R_2)\colon U_{\mathbb R^m}(0,\rho)\to (K-\widehat w)\times\mathbb R_+$
and a~constant $\gamma>0$ such that
\begin{equation}\label{l2}
\Lambda(R_1(z),R_2(z))=z,\qquad \|R_1(z)\|_X+R_2(z)\le\gamma|z|
\end{equation}
for all $z\in U_{\mathbb R^m}(0,\rho)$. This is a~particular case of an assertion
of~\cite{ArMT} proved in the case when~$X$ is finite-dimensional, but its proof can be
carried over verbatim also to this case.

The mapping $\widehat G$ is twice differentiable at the point $\widehat w$, and hence
there exists $0<\delta\le\min((8\gamma\rho)^{1/2},(8\gamma\|\widehat G''(\widehat
w)\|)^{-1},1)$ such that $U_X(\widehat w,\delta)\subset V$ and, for all $w\in
U_X(\widehat w,\delta)$,
\begin{multline}\label{l3}
|\widehat G(w)-\widehat G(\widehat w)-\widehat G'(\widehat w)(w-\widehat w)-\frac12
\widehat G''(\widehat w)[w-\widehat w,w-\widehat w]|\\\le\frac 1{16\gamma}\|w-\widehat
w\|^2_X.
\end{multline}

Let $V_1=U_{\mathbb R^m}(\widehat G(\widehat w),\delta^2/16\gamma)$. For any $y\in V_1$
we set $V_y=U_{C(V,\mathbb R^m)}(\widehat G,|y-\widehat G(\widehat w)|/4)$ (assuming with
$y=\widehat G(\widehat w)$ that $V_y=\{\widehat G\}$; in this setting the
relations~\eqref{l1} are straightforward).

Let $y\in V_1$, $y\ne\widehat G(\widehat w)$ and $G\in V_y$. Consider the mapping
$\Psi_y\colon B_{\mathbb R^m}(\widehat G(\widehat w),2|y-\widehat G(\widehat
w)|)\to\mathbb R^m$ defined by
\begin{equation*}
\Psi_y(z)=y+z-G(\widehat w+R_1(z-\widehat G(\widehat w))+(R_2(z-\widehat G(\widehat
w)))^{1/2}q).
\end{equation*}
This definition makes sense. Indeed if $z\in B_{\mathbb R^m}(\widehat G(\widehat
w),2|y-\widehat G(\widehat w)|)$, then $|z-\widehat G(\widehat w)|\le2|y-\widehat
G(\widehat w)|<2(\delta^2/16\gamma)\le\rho$. Further, $\|R_1(z-\widehat G(\widehat
w))\|_X\le\gamma|z-\widehat G(\widehat w)|\le2\gamma|y-\widehat G(\widehat
w)|<2\gamma(\delta^2/16\gamma)<\delta/2$ and $\|(R_2(z-\widehat G(\widehat
w)))^{1/2}q\|_X=(R_2(z-\widehat G(\widehat w)))^{1/2}\le(\gamma|z-\widehat G(\widehat
w)|)^{1/2}\le(2\gamma|y-\widehat G(\widehat
w)|)^{1/2}<(2\gamma)^{1/2}(\delta/4\gamma^{1/2})<\delta/2$. Hence $\widehat
w+R_1(z-\widehat G(\widehat w))+(R_2(z-\widehat G(\widehat w)))^{1/2}q\in \widehat
w+U_X(0,\delta)\subset V$.

We set for brevity, $r(z)=R_1(z-\widehat G(\widehat w))$, $\beta(z)=R_2(z-\widehat
G(\widehat w))$ and $v(z)=r(z)+(\beta(z))^{1/2}q$.

We claim that the range of the mapping $\Psi_y$ lies in the ball $B_{\mathbb
R^m}(\widehat G(\widehat w),2|y-\widehat G(\widehat w)|)$. Indeed, taking into account
the equality
$$
\widehat G'(\widehat w)v(z)+\frac12\beta(z)\widehat G''(\widehat w)[q,q]+\widehat
G(\widehat w)=z,
$$
which follows from the first relation in~\eqref{l2}, using the fact that  $q\in{\rm
Ker}\,\widehat G'(\widehat w)$, employing the elementary relation
\begin{multline*}
-\frac12\widehat G''(\widehat w)[v(z),v(z)]+\frac12\beta(z)\widehat G''(\widehat
w)[q,q]\\=-\widehat G''(\widehat w)[\frac12 r(z)+(\beta(z))^{1/2}q, \ r(z)],
\end{multline*}
invoking the inequality \eqref{l3} and since $G\in V_y$, we have
\begin{multline*} |\Psi_y(z)-\widehat G(\widehat w)|
\le|y-\widehat G(\widehat w)|+|G(\widehat w+v(z))-\widehat G(\widehat
w+v(z))|\\+|\widehat G(\widehat w+v(z))-\widehat G(\widehat w)-\widehat G'(\widehat
w)v(z)-\frac12\widehat G''(\widehat w)[v(z), \
v(z)]|\\
+|\widehat G''(\widehat w)[\frac12 r(z)+(\beta(z))^{1/2}q, r(z)]|\le|y-\widehat
G(\widehat w)|+\frac14 |y-\widehat G(\widehat
w)|\\+\frac1{16\gamma}\|v(z)\|_X^2+\|\widehat G''(\widehat w)\|\|\frac12
r(z)+(\beta(z))^{1/2}q\|_X\|r(z)\|_X.
\end{multline*}
Now $\delta\le1$, and hence, it follows from the above properties that
$2\gamma|y-\widehat G(\widehat w)|<1$ and therefore
\begin{multline*}
\|v(z)\|_X^2\le (\|r(z)\|_X+\|(\beta(z))^{1/2}q\|_X)^2\le(2\gamma|y-\widehat G(\widehat
w)|\\+(2\gamma|y-\widehat G(\widehat w)|)^{1/2})^2\le(2(2\gamma|y-\widehat G(\widehat
w)|)^{1/2})^2=8\gamma|y-\widehat G(\widehat w)|.
\end{multline*}
Further, using the same estimates,
\begin{equation*}
\|\frac12 r(z)+(\beta(z))^{1/2}q\|_X\le
\frac12\|r(z)\|_X+\|(\beta(z))^{1/2}q\|_X<\frac{\delta}4+\frac{\delta}2<\delta.
\end{equation*}
Hence
\begin{multline*}
\|\widehat G''(\widehat w)\| \,\Bigl\|\frac12
r(z)+(\beta(z))^{1/2}q\Bigr\|_X\,\|r(z)\|_X\le\|\widehat G''(\widehat w)\|
\delta2\gamma|y-\widehat G(\widehat w)|\\\le \frac14|y-\widehat G(\widehat w)|.
\end{multline*}
Combining the above estimates we arrive at the required assertion:
\begin{multline*} |\Psi_y(z)-\widehat G(\widehat w)|\le|y-\widehat
G(\widehat w)|+\frac14 |y-\widehat G(\widehat w)|+\frac12 |y-\widehat G(\widehat
w)|\\+\frac14|y-\widehat G(\widehat w)|=2|y-\widehat G(\widehat w)|.
\end{multline*}

The mapping~$\Psi_y$ is continuous \textit{qua} a~composition of continuous mappings.
Hence by Browder's fixed point there exists $\overline z=\overline z(y,G)\in B_{\mathbb
R^m}(\widehat G(\widehat w), 2|y-\widehat G(\widehat w)|)$ such that $\Psi_y(\overline
z)=\overline z$; that is, $G(\widehat w+v(\overline z))=y$. We set $w_G(y)=\widehat
w+v(\overline z)$. Therefore,  $G(w_G(y))=y$ and $\|w_G(y)-\widehat w\|_X=\|v(\overline
z)\|\le (8\gamma|y -\widehat G(\widehat w)|)^{1/2}$.  According to the above $w_G(y)\in
V$. Since $K$~is a~convex cone, we have $w_G(y)\in \widehat w +(K-\widehat w)+K=K$.
Setting $\kappa=(8\gamma)^{1/2}$, we get all the conclusions of the theorem.
\end{proof}


\begin{proof}[\bf Proof of Theorem $\ref{T1}$]
By Proposition \ref{P1} there exists a~tuple
$\widehat{\overline u}=(\widehat u _1,\ldots,\widehat u _k)\in\widehat {\mathcal U}\,^{k}$ for which \eqref{re}~holds.

It is obvious that there exist neighbourhoods $\mathcal O(\widehat \xi \,)$, $\mathcal
O(0)$, $\mathcal O(\widehat u )$ and $\mathcal O(-f(\widehat \xi ,\widehat x ))$ (of,
respectively, the point~$\widehat \xi $, the $\mathbb R^k$-origin, the point ~$\widehat u
$ and the point $-f(\widehat \xi ,\widehat x )$)  that the mapping~$\Phi$
(see~\eqref{31}) with the tuple $\overline u=\widehat{\overline u}$ is bounded on
$\mathcal O(\widehat \xi \,)\times\mathcal O(0)\times\mathcal O(-f(\widehat \xi ,\widehat
x ))\times\mathcal O(\widehat u )$.

The mappings $f$ and $g$ are continuously differentiable with respect to $(\xi,x)$ and
the mapping $(\xi,\overline\alpha,u)\mapsto x(\xi,\overline\alpha,u;\widehat{\overline
u})$ is continuously differentiable with respect to  $(\xi,\overline\alpha,u)$, and hence
reducing if necessary the neighbourhoods $\mathcal O(\widehat \xi \,)$, $\mathcal
O(\widehat x )$, $\mathcal O(0)$ and  $\mathcal O(\widehat u )$ (and assuming that they
are convex), we have from the mean value theorem the inequality
\begin{equation}\label{t1}
|f(\xi,x)-f(\xi',x')|\le c(|\xi-\xi'|+\|x-x'\|_X)
\end{equation}
for some some constant $c>0$, the inequality
\begin{equation}\label{t2}
|g(\xi,x)-g(\xi',x')|\le c(|\xi-\xi'|+\|x-x'\|_X)
\end{equation}
for all $(\xi,x)$ and $(\xi',x')$ from  $\mathcal O(\widehat \xi \,)\times\mathcal O(\widehat x )$, and the inequality
\begin{equation}\label{t3}
\|x(\xi,\overline\alpha,u;\widehat{\overline u})-\widehat x \|_X\le c(|\xi-\widehat \xi \,|+|\overline\alpha|+\|u-\widehat u \|_Z)
\end{equation}
for all $(\xi,\overline\alpha,u)\in\mathcal O(\widehat \xi \,)\times\mathcal O(0)\times\mathcal O(\widehat u )$.

By Proposition \ref{P11} there exist neighbourhoods $\mathcal O_0(\widehat \xi
\,)\subset\mathcal O(\widehat \xi \,)$, $\mathcal O_0(0)\subset\mathcal O(0)$, $\mathcal
O_0(\widehat u )\subset\mathcal O(\widehat u )$ and $\varepsilon_0>0$ such that, for any
$0<\varepsilon\le \varepsilon_0$, there exists  a~continuous mapping
$(\xi,\overline\alpha,u)\mapsto x_\varepsilon(\xi,\overline\alpha, u)$ from $\mathcal
O_0(\widehat \xi \,)\times(\mathcal O_0(0)\cap \mathbb R^k_+)\times\mathcal O_0(\widehat
u )$ into $\mathcal O(\widehat x )$ for which $F(\xi,x_\varepsilon(\xi,\overline\alpha,
u),M_\varepsilon(\overline\alpha, (u,\widehat{\overline u})))=0$ and the relation
\eqref{t4} is valid for all $(\xi,\overline\alpha, u)\in \mathcal O_0(\widehat \xi
\,)\times(\mathcal O_0(0)\cap \mathbb R^k_+)\times\mathcal O_0(\widehat u )$.

Thus, for all $0<\varepsilon\le\varepsilon_0$,
we have a~continuous mapping $\Phi_{\varepsilon}$
on
$\mathcal O_0(\widehat \xi \,)\times(\mathcal O_0(0)\cap \mathbb R^k_+)\times\mathbb R^{m_1}\times\mathcal O_0(\widehat u )$,
which assigns with a~quadruple
$(\xi,\overline\alpha,r, u)$ a~vector from~$\mathbb R^{m_1+m_2}$ by the rule
\begin{equation}\label{fim}
\Phi_{\varepsilon}(\xi,\overline\alpha,r, u)=(f(\xi,x_\varepsilon(\xi,\overline\alpha,u))+r, \
g(\xi,x_\varepsilon(\xi,\overline\alpha, u)))^T.
\end{equation}

We shall employ Theorem \ref{T3}, where $X=\mathbb R^n\times\mathbb R^k\times\mathbb
R^{m_1}\times Z$, $K=\mathbb R^n\times\mathbb R^k_+\times\mathbb R^{m_1}_+\times Z$,
$\widehat w=(\widehat \xi ,0,-f(\widehat \xi ,\widehat x ),\widehat u )$, $V=\mathcal O_0(\widehat \xi \,)\times\mathcal O_0(0)\times\mathbb
R^{m_1}\times\mathcal O_0(\widehat u )$, $\widehat G(w)=\widehat
G(\xi,\overline\alpha,r,u)=\Phi(\xi,\overline\alpha,r,u;\widehat{\overline u})$ and $q=a(\zeta,0,-\widehat
f'[\zeta,h],v)$, where $a>0$ is such that $\|q\|=1$.

It is clear that $q\in K$. Let us check that  $q\in{\rm Ker}\,\widehat G'(\widehat w)$. Indeed, we
have $(\zeta,h,v)\in K(\widehat \xi ,\widehat x ,\widehat u )$, and hence $h=-\widehat F^{-1}_x\widehat F_\xi \zeta-\widehat F^{-1}_x\widehat F_u
v=\widehat x _\xi \zeta+\widehat x _u v$ and   $\widehat f_\xi \zeta+\widehat f_x\widehat x _\xi
\zeta+\widehat f_x \widehat x _{\overline\alpha}0+\widehat f_x\widehat x _u v-\widehat f'[\zeta,h]=\widehat
f_\xi \zeta+\widehat f_xh-\widehat f'[\zeta,h]=\widehat f'[\zeta,h]-\widehat
f'[\zeta,h]=0$. In a~similar manner, $\widehat g_\xi \zeta+\widehat g_x\widehat x _\xi \zeta+\widehat g_x
\widehat x _{\overline\alpha}0+\widehat g_x\widehat x _u v=\widehat g'[\zeta,h]=0$. These equalities show that
$q\in{\rm Ker}\,\widehat G'(\widehat w)$.

Inclusion \eqref{re} implies \eqref{l} (in our setting), because in \eqref{l} the
bracketed set in wider than the corresponding set in~\eqref{re}.

It is clear that $\widehat G(\widehat w)=0$. Let the neighbourhood $V_1$ of the origin in $\mathbb
R^{m_1}\times\mathbb R^{m_2}$ and the constant $\kappa>0$ be as in Theorem~\ref{T3}.

Further, let $W$ be an arbitrary neighbourhood of $(\widehat \xi ,\widehat x )$ and $\rho>0$ such that
$U_{\mathbb R^n\times X}((\widehat \xi ,\widehat x ),\rho)\subset W$. We set $\kappa_0=1+(c+1)\kappa$ ($c$~is
the constant in the inequalities \eqref{t1}--\eqref{t3}). Let $0<r\le\rho^2/\kappa_0^2$ and
let $W_1$ and $W_2$ be neighbourhoods of the $\mathbb R^{m_1}$- and $\mathbb R^{m_2}$-origins
such that $W_1\times W_2\subset U_{\mathbb R^{m_1+m_2}}(0,r)\subset V_1$.

Let $y=(y_1,y_2)\in W_1\times W_2$ and let $V_y$ be the corresponding neighbourhood from Theorem~\ref{T3}.
From \eqref{t1}, \eqref{t2} and \eqref{t4} it follows that there exists
$\varepsilon=\varepsilon(y)\le|y|^{1/2}/2\|\widehat F_x^{-1}\|$ such that $\Phi_{\varepsilon}\in
V_y$. By this theorem there exists a~point $(\xi_y,\overline\alpha_y,r_y,u_y)\in V\cap
K$ for which
\begin{equation}\label{l6}
f(\xi_y,x_\varepsilon(\xi_y,\overline\alpha_{y}, u_y))+r_{y}=y_1, \quad
g(\xi_y,x_\varepsilon(\xi_y,\overline\alpha_{y}, u_y))=y_2
\end{equation}
and
\begin{equation}\label{l5}
|\xi_y-\widehat \xi \,|+ |\overline\alpha_y|+|r_y+f(\widehat \xi ,\widehat x )|+\|u_y-\widehat u \|_Z\le \kappa|y|^{1/2}.
\end{equation}

By  \eqref{t4}, \eqref{t3}, \eqref{l5} and by the choice of $\varepsilon$ we have
$\|x_\varepsilon(\xi_y,\overline\alpha_y,u_y)-\widehat x \|_X+ |\xi_y-\widehat \xi \,|\le
\|x_\varepsilon(\xi_y,\overline\alpha_y,u_y)-x(\xi_y,\overline\alpha_y,u_y;\widehat{\overline u})\|_X+\|x(\xi_y,\overline\alpha_y,u_y;\widehat{\overline u})-\widehat x \|_X+
|\xi_y-\widehat \xi \,|<|y|^{1/2}+c\kappa|y|^{1/2}+\kappa|y|^{1/2}=\kappa_0|y|^{1/2}$.

We set $x_y=x_\varepsilon(\xi_y,\overline\alpha_y,u_y)$ and
$u_y=M_{\varepsilon}(\overline\alpha_y,(u_y,\widehat{\overline u}))$. Then $F(\xi_y,x_y,u_y)=0$.
Since $r_y\ge0$ from~\eqref{l6} it follows that  $f(\xi_y,x_y)\le y_1$ and that $g(\xi_y,x_y)= y_2$.
These inequalities imply that $\|x_y-\widehat x \|_X+ |\xi_y-\widehat \xi \,|\le\kappa_0|y|^{1/2}$.
We have $\kappa_0|y|^{1/2}<\rho$, and hence $(x_y,\xi_y)\in W$.
\end{proof}

Now as another simple corollary to Theorem \ref{T1}
we have the following second-order necessary conditions for a~strong minimum in the following
abstract optimal control problem
\begin{multline}\label{p}
f_0(\xi,x)\to\min,\quad F(\xi,x,u)=0, \quad u\in\mathcal U,\quad f(\xi,x)\le0,\\
g(\xi,x)=0,
\end{multline}
where the set $\mathcal U$ and the mappings $F$, $f$ and $g$
are the same as in the definition of the control system \eqref{1}, a~function
$f_0\colon \mathbb R^n\times X\to\mathbb R$ is also given.

An admissible point $(\widehat \xi ,\widehat x ,\widehat u )$ point for this problem (that is, $(\widehat \xi ,\widehat x ,\widehat u )$ satisfies the
constraints of the problem) is called a~{\it strong minimum} if there exists a~neighbourhood
$W$ of $(\widehat \xi ,\widehat x )$ such that $f_0(\xi,x)\ge f_0(\widehat \xi ,\widehat x )$ for all admissible points
$(\xi,x,u)\in W\times \mathcal U$.

With problem \eqref{p} we shall associate the Lagrange function
\begin{equation*}
\mathcal L(\xi,x,u,\overline\lambda)=\lambda_0f_0(\xi,x)+\langle
y^*,F(\xi,x,u)\rangle+\langle\lambda_1,f(\xi,x)\rangle+\langle\lambda_2,g(\xi,x)\rangle,
\end{equation*}
where $\overline\lambda=(\lambda_0,y^*,\lambda_1,\lambda_2)\in\mathbb R\times Y^*\times(\mathbb
R^{m_1})^*\times(\mathbb R^{m_2})^*$.

For an admissible point $(\widehat \xi ,\widehat x ,\widehat u )$ for problem for which  $\widehat u \in\operatorname{int}\mathcal U$, we define the set
\begin{multline*}
K_0(\widehat \xi ,\widehat x ,\widehat u )=\{\,q=(\zeta,h,v)\in \mathbb R^n\times X\times Z : \widehat F'q=0,\quad
\widehat f'_0[\zeta,h]\le0,\\ \widehat f'[\zeta,h]\le0,\quad \widehat g'[\zeta,h]=0\,\}.
\end{multline*}

\begin{corollary}[\bf Second-order minimum conditions for problem $\eqref{p}$]\label{S1}
Let  $(\widehat \xi ,\widehat x ,\widehat u )$, where $\widehat u \in\operatorname{int}\mathcal U$, be a~strong minimum point
in problem~\eqref{p}. Then if the Basic Assumptions are satisfied\,\footnote{We naturally assume that $f_0$
features the same properties as $f$ and~$g$.}, then, for any  $q=(\zeta,h,v)\in K_0(\widehat \xi ,\widehat x ,\widehat u )$,
there exists a~nonzero tuple
$\lambda=\lambda(q)=(\lambda_0,\lambda_1,\lambda_2)\in\mathbb R_+\times(\mathbb
R^{m_1})_+^*\times(\mathbb R^{m_2})^*$ and a~functional $y^*=y^*(q)\in Y^*$ such that
\begin{equation*}
\begin{aligned}
\mathcal L_\xi(\widehat \xi ,\widehat x ,\widehat u ,\overline\lambda)=0\,\,\Leftrightarrow\,\,\lambda_0\widehat
f_{0\xi}+\widehat F_\xi^*y^*+
\widehat f_\xi^*\lambda_1+{\widehat g_\xi}^*\lambda_2&=0,\\[5pt]
\mathcal L_x(\widehat \xi ,\widehat x ,\widehat u ,\overline\lambda)=0\,\,\Leftrightarrow\,\,\lambda_0\widehat
f_{0x}+\widehat F_x^*y^*+\widehat f_x^*\lambda_1+\widehat g_x^*\lambda_2&=0,
\end{aligned}
\end{equation*}
\begin{equation*}
\langle\lambda_1,f(\widehat \xi ,\widehat x )\rangle=0,
\end{equation*}
\begin{multline*}
\min_{u\in\mathcal U}\mathcal L(\widehat \xi ,\widehat x ,u,\overline\lambda)=\mathcal
L(\widehat \xi ,\widehat x ,\widehat u ,\overline\lambda)\\\Leftrightarrow\,\,\,\min_{u \in \mathcal U}\langle y^*,
F(\widehat \xi ,\widehat x ,u)\rangle=\langle y^*,F(\widehat \xi ,\widehat x ,\widehat u )\rangle=0
\end{multline*}
and
\begin{multline*}
\mathcal L_{(\xi,x,u)(\xi,x,u)}(\widehat \xi ,\widehat x ,\widehat u ,\overline\lambda)[q,q]\ge0\,\,\Leftrightarrow\,\,
\langle y^*,\widehat F''[q,q]\rangle+\lambda_0\widehat
f_0''[(\zeta,h),(\zeta,h)]\\+\langle\lambda_1,\widehat
f''[(\zeta,h),(\zeta,h)]\rangle+\langle\lambda_2,\widehat g''[(\zeta,h),(\zeta,h)]\rangle\ge0.
\end{multline*}

If, for the control system specifying the constraints in problem in problem~\eqref{p},
$\Lambda(\widehat \xi ,\widehat x ,\widehat u ,q)=\emptyset$ for some  $q\in K_0(\widehat \xi ,\widehat x ,\widehat u )$,
then $\lambda_0\ne0$.
\end{corollary}

\begin{proof} The proof is by \textit{reductio ad absurdum}.
Assume that there exists $q=(\zeta,h,v)\in K_0(\widehat \xi ,\widehat x ,\widehat u )$ such that only the
tuples $\overline\lambda=(\lambda_0,y^*,\lambda_1,\lambda_2)$, where
$(\lambda_0,\lambda_1,\lambda_2)=0$, satisfy all the constraints in the assertion of the theorem.
This means that if one considers the control system
\begin{multline*}
F(\xi,x,u)=0,\quad u\in \mathcal U,\quad f_0(\xi,x)-f_0(\widehat \xi ,\widehat x )\le0,\quad f(\xi,x)\le0,\\
g(\xi,x)=0
\end{multline*}
and denotes by $\Lambda_1(\widehat \xi ,\widehat x ,\widehat u ,q)$ the analogue of the set
$\Lambda(\widehat \xi ,\widehat x ,\widehat u ,q)$ for this system, then $\Lambda_1(\widehat \xi ,\widehat x ,\widehat u ,q)=\emptyset$.
Hence by Theorem~\ref{T1} this system is locally controllable with respect to the point
$(\widehat \xi ,\widehat x ,\widehat u )$.

Let $W$ be an arbitrary neighbourhood of $(\widehat \xi ,\widehat x )$ and $W_1$, $W_2$
be the corresponding neighbourhoods of the origins in $\mathbb R^{m_1+1}$ and $\mathbb R^{m_2}$
from the definition of controllability. It is clear that
$y(\varepsilon)=((-\varepsilon,0),0)\in W_1\times W_2$ for sufficiently small $\varepsilon>0$.
Hence, by the local controllability, for any such~$\varepsilon$ there exists an element
$(\xi_{y(\varepsilon)},x_{y(\varepsilon)},u_{y(\varepsilon)})\in W\times \mathcal U$
for which $F(\xi_{y(\varepsilon)},x_{y(\varepsilon)},u_{y(\varepsilon)})=0$, \
$f_0(\xi_{y(\varepsilon)},x_{y(\varepsilon)})\le f_0(\widehat \xi ,\widehat x )-\varepsilon$, \
$f(\xi_{y(\varepsilon)},x_{y(\varepsilon)})\le 0$,
$g(\xi_{y(\varepsilon)},x_{y(\varepsilon)})=0$ and
$(\xi_{y(\varepsilon)},x_{y(\varepsilon)})\in W$, contradicting the fact that
$(\widehat \xi ,\widehat x ,\widehat u )$ is a~strong minimum point for problem~\eqref{p}.
\end{proof}

\section{Application to control dynamical systems}

Let $[t_0,t_1]$ be a closed interval  of the real line, $U$~be
an open subset of~$\mathbb R^r$,
$\varphi\colon \mathbb R\times\mathbb R^n\times U\to \mathbb R^n$ be a~mapping
of the variables $t\in \mathbb R$, $x\in\mathbb R^n$ and $u\in U$, and let $f\colon\mathbb
R^n\times\mathbb R^n\to \mathbb R^{m_1}$ and $g\colon\mathbb R^n\times\mathbb R^n\to
\mathbb R^{m_2}$ be mappings of the variables $\zeta_i\in\mathbb R^n$, $i=1,2$.

Let us consider a control dynamical system
\begin{multline}\label{r1} \dot x
=\varphi(t,x,u),\quad u(t)\in U \,\,\,\text{for almost all}\,\,\, t\in[t_0,t_1],\\
f(x(t_0),x(t_1))\le0,\quad g(x(t_0),x(t_1))=0,
\end{multline}
where $x(\cdot)\in AC([t_0,t_1],\mathbb R^n)$ (absolutely continuous vector functions on
$[t_0,t_1]$) and $u(\cdot)\in L_\infty([t_0,t_1],\mathbb R^r)$.

A pair $(\widehat x (\cdot),\widehat u (\cdot))$ will be called an admissible process for this
system if it satisfies all the constraints and there exists a~compact set $K\subset U$
such that $\widehat u (t)\in
K$ for almost all $t\in[t_0,t_1]$.

A control $u(\cdot)$ with the above property will be called {\it regular}.

\begin{definition}\label{def4}
\rm A~system \eqref{r1} will be said to be \textit{locally controllable} with respect to an admissible process
$(\widehat x (\cdot),\widehat u (\cdot))$ if, for each neighbourhood~$W$ of the point~$\widehat x (\cdot)$, there exist
neighbourhoods~$W_1$ and $W_2$ of the origins in~$\mathbb R^{m_1}$ and $\mathbb R^{m_2}$,
respectively, such that, for any
$y=(y_1,y_2)\in W_1\times W_2$ there exists a~pair $(x_y(\cdot),u_y(\cdot))\in AC([t_0,t_1],\mathbb
R^n)\times L_\infty([t_0,t_1],\mathbb R^r)$ satisfying the conditions:
$\dot x_y(t) =\varphi(t,x_y(t),u_y(t))$ and $u_y(t)\in U$ for almost all $t\in[t_0,t_1]$ and which
is such that $x_y(\cdot)\in W$, \ $f(x_y(t_0),x_y(t_1))\le y_1$ and $g(x_y(t_0),x_y(t_1))=y_2$.
\end{definition}

In what follows we assume that the {\it mapping $\varphi$ is continuous together with
its second derivative with respect to $(x,u)$ on $\mathbb R\times\mathbb R^n\times U$
and the mappings~$f$
and~$g$ have continuous second derivatives on $\mathbb R^n\times\mathbb R^n$}.

Given a fixed admissible process $(\widehat x (\cdot),\widehat u (\cdot))$ for system \eqref{r1}, the derivatives
of the mappings~$f$ and~$g$ at the point $(\widehat x (t_0),\widehat x (t_1))$
will be briefly denoted by $\widehat f'$ and $\widehat g'$, their partial derivatives with respect to~$\zeta_1$ and
$\zeta_2$ at the point $(\widehat x (t_0),\widehat x (t_1))$ will be written, respectively, as
$\widehat f_{\zeta_i}$ and $\widehat g_{\zeta_i}$, $i=1,2$.
The adjoint operators will be denoted, respectively, by ${\widehat
{f}_{\zeta_i}}^*$ and ${\widehat {g}_{\zeta_i}}^*$. We shall also write $\widehat
\varphi(t)=\varphi(t,\widehat x (t),\widehat u (t))$, and similarly, for the derivatives $\widehat
\varphi_x(t)=\varphi_x(t,\widehat x (t),\widehat u (t))$ and $\widehat
\varphi_u(t)=\varphi_u(t,\widehat x (t),\widehat u (t))$.

We set $H(t,x,u,p(\cdot))=\langle p(t),\varphi(t,x,u)\rangle$, where $p(\cdot)\colon
[t_0,t_1]\to(\mathbb R^n)^*$.

For brevity, we shall write $w=(x,u)$ and $\eta=(h(t_0),h(t_1))$ if $h(\cdot)\in
C([t_0,t_1],\mathbb R^n)$.

Let $(\widehat x (\cdot),\widehat u (\cdot))$ be an admissible process  for system \eqref{r1}. For any pair
$q(\cdot)=(h(\cdot), v(\cdot))\in C([t_0,t_1],\mathbb R^n)\times L_\infty([t_0,t_1],\mathbb R^r)$
we consider the following system of relations with respect to the variables $p(\cdot)\in
AC([t_0,t_1],(\mathbb R^n)^*)$, $\lambda_1\in(\mathbb R^{m_1})^*_+$ and
$\lambda_2\in(\mathbb R^{m_2})^*$:
\begin{equation}\label{lam}
\begin{cases}
-\dot p=p\,\widehat \varphi_x(t),\quad p(t_0)={\widehat
{f}_{\zeta_1}}^*\lambda_1+{\widehat {g}_{\zeta_1}}^*\lambda_2,\quad p(t_1)=-{\widehat
{f}_{\zeta_2}}^*\lambda_1\\[3pt]-{\widehat {g}_{\zeta_2}}^*\lambda_2;\\[7pt]
\max\limits_{u\in U}H(t,\widehat x (t),u,p(t))=H(t,\widehat x (t),\widehat u (t),p(t))\,\,\text {for
a.a.}\\[3pt] t\in[t_0,t_1];\\[9pt]
\langle\lambda_1,f(\widehat x(t_0),\widehat x(t_1))\rangle=0;\\[9pt]
-\displaystyle\int_{t_0}^{t_1} H_{ww}(t,\widehat x (t),\widehat u (t),p(t))[q(t),q(t)]\,dt+
\langle\lambda_1,\widehat f''[\eta,\eta]\rangle\\[10pt]+\langle\lambda_2,\widehat
g''[\eta,\eta]\rangle\ge0.
\end{cases}
\end{equation}

We let $\Lambda(\widehat x (\cdot),\widehat u (\cdot),q(\cdot))$ denote the set of triples
$(p(\cdot),\lambda_1,\lambda_2)\in AC([t_0,t_1],(\mathbb R^n)^*)\times(\mathbb
R^{m_1})^*_+\times(\mathbb R^{m_2})^*$, which satisfy all the relations in
\eqref{lam} and which are such that $|\lambda_1|+|\lambda_2|\ne0$.

Let $(\widehat x (\cdot),\widehat u (\cdot))$ be an admissible process for system \eqref{r1}. We set
\begin{multline*}
K(\widehat x (\cdot),\widehat u (\cdot))=\{\,q(\cdot)=(h(\cdot),v(\cdot))\in AC([t_0,t_1],\mathbb R^n)\times
L_\infty([t_0,t_1],\mathbb R^r) \\ : \dot h(t)=\widehat
\varphi_x(t)h(t)+\widehat\varphi_u(t)v(t),\quad \widehat f'[h(t_0),h(t_1)]\le0,\\\widehat
g'[h(t_0),h(t_1)]=0\,\}.
\end{multline*}

\begin{theorem}\label{T2}
Let $(\widehat x (\cdot),\widehat u (\cdot))$ be an admissible process for system \eqref{r1}. Assume that there exists
$q(\cdot)=(h(\cdot),v(\cdot))\in K(\widehat x (\cdot),\widehat u (\cdot))$ such that $\Lambda(\widehat x (\cdot),\widehat u (\cdot),q(\cdot))=\emptyset$.
Then system \eqref{r1} is locally  controllable with respect to the process $(\widehat x (\cdot),\widehat u (\cdot))$.

Moreover, there exists a~constant $c_0>0$ such that
$\|x_y(\cdot)-\widehat x (\cdot)\|_{C([t_0,t_1],\mathbb R^n)}\le c_0|y|^{1/2}$
for the variables $y$ and $x_y(\cdot)$ from the definition of controllability
of system~\eqref{r1}
\end{theorem}

\begin{proof} With a~control dynamical system \eqref{r1} we shall associate a~control system of the form~\eqref{1}.
Let $X=Y=C([t_0,t_1],\mathbb R^n)$, $Z=L_\infty([t_0,t_1],\mathbb R^r)$ and
$\mathcal U=\{\,u(\cdot)\in L_\infty([t_0,t_1],\mathbb R^r) : u(t)\in U$ for almost all
$t\in[t_0,t_1]\,\}$. We define the mapping $F\colon \mathbb R^n\times X\times\mathcal U\to
Y$ by the formula
\begin{equation*}
F(\xi,x(\cdot),u(\cdot))(t)=-\xi+x(t)-\int_{t_0}^t\varphi(\tau,x(\tau),u(\tau))\,d\tau,\quad
 \forall\,\,t\in[t_0,t_1].
\end{equation*}
The mappings  $f$ and $g$ in~\eqref{r1} will be considered as mappings $f\colon \mathbb
R^n\times X\to \mathbb R^{m_1}$ and $g\colon \mathbb R^n\times X\to \mathbb R^{m_2}$,
which associate with each pair $(\xi,x(\cdot))$ the vectors $f(\xi,x(t_1))$ and
$g(\xi,x(t_1))$, respectively.

Let us consider the control system
\begin{multline}\label{p11}
F(\xi,x(\cdot),u(\cdot))(\cdot)=0, \quad u(\cdot)\in \mathcal U, \quad f(\xi,x(\cdot))\le0,\\
g(\xi,x(\cdot))=0,
\end{multline}
which looks like system \eqref{1}.

If $(\widehat x (\cdot),\widehat u (\cdot))$ is an admissible process for system \eqref{r1} (and hence the
control $\widehat u (\cdot)$ is regular), then it is easily checked that in this case
$\widehat u (\cdot)\in\operatorname{int}\mathcal U$. Hence the point $(\widehat x (t_0),\widehat x (\cdot),\widehat u (\cdot))$ is admissible for
the control system \eqref{p11}.

System \eqref{p11} satisfies the Basic Assumptions. Indeed,
$1)$ clearly holds. Further, standard arguments show  that the assumptions on
the mappings in system~\eqref{r1} guarantee condition~$2)$, and besides,
the operator $F_{x(\cdot)}(\widehat \xi ,\widehat x (\cdot),\widehat u (\cdot))$ is well-known to be invertible.
That $3)$~holds was proved in~\cite{AM} (under weaker assumptions).

By the hypothesis, $q(\cdot)=(h(\cdot),v(\cdot))\in K(\widehat x (\cdot),\widehat u (\cdot))$.
It easily follows that the triple
$q_1(\cdot)=(h(t_0),h(\cdot),v(\cdot))$ lies in the cone~\eqref{k}, which was written down for system
\eqref{p11}.

We let $\Lambda(\widehat \xi ,\widehat x (\cdot),\widehat u (\cdot),q_1(\cdot))$ denote the set of triples
$(y^*,\lambda_1,\lambda_2)\in Y^*\times(\mathbb R^{m_1})^*_+\times(\mathbb R^{m_2})^*$,
$|\lambda_1|+|\lambda_2|\ne0$, which satisfy the relations in~\eqref{reg},
as written for system~\eqref{p11}. We claim that
$\Lambda(\widehat \xi ,\widehat x (\cdot),\widehat u (\cdot),q_1(\cdot))=\emptyset$ under the hypotheses of the theorem.

Indeed, in~\cite{AM} it was shown that if a~tuple $(y^*,\lambda_1,\lambda_2)\in
Y^*\times(\mathbb R^{m_1})^*_+\times(\mathbb R^{m_2})^*$, where
$|\lambda_1|+|\lambda_2|\ne0$, satisfies the equalities in~\eqref{reg}, then there exists
$p(\cdot)\in AC([t_0,t_1],(\mathbb R^n)^*)$ such that the tuple $(p(\cdot),\lambda_1,\lambda_2)$,
satisfies the equalities  in~\eqref{lam}.

Let us now show that if $(y^*,\lambda_1,\lambda_2)$ also satisfies the inequality in~\eqref{reg},
then $(p(\cdot),\lambda_1,\lambda_2)$ satisfies the inequality in~\eqref{lam}. According to~\cite{AM},
the functional $y^*$, \textit{qua} a~linear continuous functional on
$C([t_0,t_1], \mathbb R^n)$, is defined by a~function of bounded variation~$\mu(\cdot)$, which
is related with the function $p(\cdot)$ via the relation: $p(t)=\mu(t_1)-\mu(t)$ if $t\in[t_0,t_1)$ and
$p(t_1)=-{\widehat {f}_{\zeta_2}}^*\lambda_1-{\widehat {g}_{\zeta_2}}^*\lambda_2$.
Now if $q(\cdot)=(h(\cdot),v(\cdot))$ and $w=(x,u)$, then we have, by changing the order of integration,
\begin{multline*}
\langle y^*,\widehat F''[q_1(\cdot),q_1(\cdot)]\rangle=\int_{t_0}^{t_1}
\left(-\int_{t_0}^t\widehat\varphi_{ww}(\tau)[q(\tau),q(\tau)]\,d\tau\right)d\mu(t)\\
=-\int_{t_0}^{t_1}\langle p(t),\widehat\varphi_{ww}(t)[q(t),q(t)]\rangle\,dt\\=-\int_{t_0}^{t_1}
H_{ww}(t,\widehat x (t),\widehat u (t),p(t))[q(t),q(t)]\,dt.
\end{multline*}
Next it is clear that if $\eta=(h(t_0),h(t_1))$, then the second and third terms on the left of
the inequality in~\eqref{reg} are of the form $\langle\lambda_1,\widehat
f''[\eta,\eta]\rangle$ and $\langle\lambda_2,\widehat g''[\eta,\eta]\rangle$; that is,
$(p(\cdot),\lambda_1,\lambda_2)$ satisfies inequality in~\eqref{lam}, thereby showing that
$(p(\cdot),\lambda_1,\lambda_2)\in\Lambda(\widehat x (\cdot),\widehat u (\cdot),q(\cdot))$.

Thus if $\Lambda(\widehat x (\cdot),\widehat u (\cdot),q(\cdot))=\emptyset$, then $\Lambda(\widehat \xi ,\widehat x (\cdot),\widehat u (\cdot),q_1(\cdot))=\emptyset$.
Hence, by Theorem~\ref{T1} system~\eqref{p11} is locally controllable with respect to the point $(\widehat x (t_0),\widehat x (\cdot),\widehat u (\cdot))$.
This readily implies the local controllability of system~\eqref{r1} with respect to the process
$(\widehat x (\cdot),\widehat u (\cdot))$.
\end{proof}

From this theorem, as in the abstract setting, we immediately derive the second-order
necessary conditions for the following optimal control problem
\begin{multline}\label{z}
f_0(x(t_0),x(t_1))\to\min,\quad \dot x =\varphi(t,x,u),\quad u(t)\in U,\\
f(x(t_0),x(t_1))\le0,\quad g(x(t_0),x(t_1))=0,
\end{multline}
where the set $U$ and the mappings $\varphi$, $f$ and~$g$ are the same as in system~\eqref{r1},
the function $f_0$ is defined on $\mathbb R^n\times\mathbb R^n$ and has the same properties as
 $f$ and~$g$.

A point $(\widehat x (\cdot),\widehat u (\cdot))$ admissible for this problem is called a~{\it strong minimum point} if
there exists a~neighbourhood of the function $\widehat x (\,\cdot\,)$ in $C([t_0,t_1],\mathbb R^n)$ such that,
for all admissible points $(x(\cdot),u(\cdot))$ with $x(\cdot)$  from this neighbourhood,
the inequality $f_0(x(t_0),x(t_1))\ge f_0(\widehat x (t_0),\widehat x (t_1))$ holds.

We set
\begin{multline*}
K_0(\widehat x (\cdot),\widehat u (\cdot))=\{\,q(\cdot)=(h(\cdot),v(\cdot))\in AC([t_0,t_1],\mathbb R^n)\\\times
L_\infty([t_0,t_1],\mathbb R^r)  : \dot h(t)=\widehat
\varphi_x(t)h(t)+\widehat\varphi_u(t)v(t),\\\widehat
f'_0[h(t_0),h(t_1)]\le0,\quad\widehat f'[h(t_0),h(t_1)]\le0,\quad\widehat
g'[h(t_0),h(t_1)]=0\,\}.
\end{multline*}

\begin{corollary}[\bf Second-order minimum conditions for problem $\eqref{z}$]\label{S2}
If $(\widehat x (\cdot),\widehat u (\cdot))$ is a~strong minimum point in problem $\eqref{z}$, then for any  $q(\cdot)\in
K_0(\widehat x (\cdot),\widehat u (\cdot))$ there exists a~nonzero tuple $(\lambda_0,\lambda_1,\lambda_2)\in \mathbb
R_+\times(\mathbb R^{m_1})^*_+\times(\mathbb R^{m_2})^*$ and a~function $p(\cdot)\in
AC([t_0,t_1],(\mathbb R^n)^*)$ such that
\begin{equation*}\begin{cases}
-\dot p=p\,\widehat \varphi_x(t),\quad p(t_0)=\lambda_0{\widehat
{f}}_{0\zeta_1}+{\widehat {f}_{\zeta_1}}^*\lambda_1+{\widehat
{g}_{\zeta_1}}^*\lambda_2,\\[7pt] p(t_1)=-\lambda_0{\widehat
{f}}_{0\zeta_2}-{\widehat
{f}_{\zeta_2}}^*\lambda_1-{\widehat {g}_{\zeta_2}}^*\lambda_2;\\[7pt]
\max\limits_{u\in U}H(t,\widehat x (t),u,p(t))=H(t,\widehat x (t),\widehat u (t),p(t)),\,\,\,\text {for a.a.}\,\,
t\in[t_0,t_1];\\[10pt]
\langle\lambda_1,f(\widehat x(t_0),\widehat x(t_1))\rangle=0;\\[7pt]
-\displaystyle\int_{t_0}^{t_1}
H_{ww}(t,\widehat x (t),\widehat u (t),p(t))[q(t),q(t)]\,dt+\lambda_0\widehat f''_0[\eta,\eta]\\[10pt]+
\langle\lambda_1,\widehat f''[\eta,\eta]\rangle+\langle\lambda_2,\widehat g''[\eta,\eta]\rangle\ge0.
\end{cases}
\end{equation*}

If, for the system
specifying the constraints in problem \eqref{z},   $\Lambda(\widehat x (\cdot),\widehat u (\cdot),
q(\cdot))=\emptyset$ for some  $q(\cdot)\in K_0(\widehat x (\cdot),\widehat u (\cdot))$, then $\lambda_0\ne0$.
\end{corollary}

The proof of this corollary is the same as that for Corollary~\ref{S1}.

\vskip20pt

Let us give some comments on the results obtained in this section.
Together with relations \eqref{lam} we shall also consider the relation
\begin{equation}\label{ns}
H_u(t,\widehat x (t),\widehat u (t),p(t))=0,\,\,\,\,\text{for a.a.}\,\, t\in[t_0,t_1].
\end{equation}

We let $\Lambda_{\mathrm{max}}(\widehat x (\cdot),\widehat u (\cdot))$  the set of triples
$(p(\cdot),\lambda_1,\lambda_2)\in AC([t_0,t_1]$, $(\mathbb R^n)^*)\times(\mathbb
R^{m_1})^*_+\times(\mathbb R^{m_2})^*$, $|\lambda_1|+|\lambda_2|\ne0$, satisfying
all the relations in~\eqref{lam} except for the last inequality. We also denote by
$\Lambda(\widehat x (\cdot),\widehat u (\cdot))$ the set of similar triples, but with the maximum condition replaced by
condition~\eqref{ns}.

The local controllability of system \eqref{r1} with respect to admissible process
$(\widehat x (\cdot), \widehat u (\cdot))$ in the case of an open~$U$ is well-known to
follow from the complete controllability of a~linear approximation to this system in
a~neighbourhood of this point. Kalman (see, for example, \cite{LM}) seems to be the first
to prove this fact. In our terms this is equivalent to saying that
\begin{equation}\label{7c}
\Lambda(\widehat x (\cdot),\widehat u (\cdot))=\emptyset.
\end{equation}
In \cite{AM} it was shown, in particular, that the local controllability takes place
also under weaker assumptions;  namely, when
\begin{equation}\label{8c}
\Lambda_{\mathrm{max}}(\widehat x (\cdot),\widehat u (\cdot))=\emptyset,
\end{equation}
where $U$ is an arbitrary set and a~control $\widehat u (\cdot)$ not necessarily regular
(if $U$ is open and $\widehat u (\cdot)$ is a~regular control, then this readily follows from Theorem~\ref{T2}
with $q(\cdot)=0$). A~similar result can also be derived from the maximum principle for the
geometric optimal control problem; see~\cite{AS}. The paper~\cite{P}
puts forward conditions for local controllability for a~dynamical system with
fixed end-points, which can be looked upon as sufficient conditions that
\eqref{8c} holds in the setting $\widehat x (\cdot)=0$, $\widehat u (\cdot)=0$ (even though the local controllability in~\cite{P}
is understood in a~somewhat more general sense: the point $(0,0)$ is not
assumed to be admissible for the corresponding control system).

Theorem \ref{T2} gives sufficient conditions for local controllability in the setting when
relations \eqref{7c} and/or \eqref{8c} may fail to hold.
Problems of local controllability for dynamical systems linear in the control
were extensively studied in a~similar setting.
In this case, if the set $U$~is open, then conditions \eqref{7c} and
\eqref{8c} are clearly equivalent.
The most comprehensive account on necessary and sufficient conditions for local controllability for such problems
may be found in~\cite{S} (see also
the references cited in~\cite{AS}). The character of such conditions is different from that given by Theorem~\ref{T2}.
We also note the paper~\cite{AJ}, which puts forward necessary and sufficient conditions for
local controllability under the condition of 2-normality
of a~dynamical system (which was introduced in~\cite{AJ}). These conditions make sense for problems
when condition \eqref{7c}  fails to hold.

The second-order necessary conditions for optimality for the optimal control
problem \eqref{p}, as given as a~direct corollary to Theorem~\ref{T2},
are similar to those form~\cite{O}, but which were obtained without the assumption about the
piecewise continuity of the optimal control itself.

\section{Control dynamical systems of the first order of abnormality. Examples}

Let $(\widehat x (\cdot),\widehat u (\cdot))$ be an admissible process for system \eqref{r1}. The set
$\Lambda_{\mathrm{max}}(\widehat x (\cdot),\widehat u (\cdot))$, when nonempty, is a~normed cone in
a~finite-dimensional space. A~system \eqref{r1} with respect to the process
$(\widehat x (\cdot),\widehat u (\cdot))$ will be said to have the $k$th {\it order of abnormailty $k\in\mathbb N$} if the
dimension of the linear hull of $\Lambda_{\mathrm{max}}(\widehat x (\cdot),\widehat u (\cdot))$ is~$k$.

This definition can be looked upon as an extension of the order of
abnormality, which was introduced by Bliss~see \cite{B}).

Given a~normed linear space, a~cone containing the origin of the space is called a~{\it pointed cone} if
it does not contain any proper subspace.

An admissible process $(\widehat x (\cdot),\widehat u (\cdot))$ is called {\it singular} if the cone
$\Lambda_{\mathrm{max}}(\widehat x (\cdot),\widehat u (\cdot))\cup\{0\}$ is not pointed.

This is easily seen to be equivalent to the definition that there exists a~nonzero
triple $(p(\cdot),0,\lambda_2)\in\Lambda_{\mathrm{max}}(\widehat x (\cdot),\widehat u (\cdot))$ such that the mapping
$u\mapsto H(t,\widehat x (t),u,p(t))$ is constant for almost all $t\in[t_0,t_1]$.

In the case when the order of abnormailty is equal to one, we give one corollary to Theorem~\ref{T2}
which is useful in applications.

We let $Q(p(\cdot),\lambda_1,\lambda_2)[q(\cdot),q(\cdot)]$ denote the expression on the left in the last
inequality in~\eqref{lam}.

\begin{corollary}\label{S3}
Assume that system \eqref{r1} with respect to a~nonsingular admissible process $(\widehat x (\cdot),\widehat u (\cdot))$
has the first order of abnormailty and that $(p(\cdot),\lambda_1,\lambda_2)\in
\Lambda_{\mathrm{max}}(\widehat x (\cdot),\widehat u (\cdot))$. If there exists an element  $q(\cdot)\in K(\widehat x (\cdot),\widehat u (\cdot))$
such that $Q(p(\cdot),\lambda_1,\lambda_2)[q(\cdot),q(\cdot)]<0$, then system \eqref{r1}
is locally controllable with respect to the process $(\widehat x (\cdot),\widehat u (\cdot))$.
\end{corollary}
\begin{proof}[Proof]
Assume that system \eqref{r1} is not locally controllable with respect to the process
$(\widehat x (\cdot),\widehat u (\cdot))$. Then $\Lambda(\widehat x (\cdot),\widehat u (\cdot),q'(\cdot))\ne\emptyset$ for any $q'(\cdot)\in
K(\widehat x (\cdot),\widehat u (\cdot))$. Let $(p'(\cdot),\lambda'_1,\lambda'_2)\in\Lambda(\widehat x (\cdot),\widehat u (\cdot),q'(\cdot))$.
Since the order of abnormailty is~$1$ and since the pair $(\widehat x (\cdot),\widehat u (\cdot))$
is nonsingular, we have
$(p(\cdot),\lambda_1,\lambda_2)=\alpha(p'(\cdot),\lambda'_1,\lambda'_2)$ for some $\alpha>0$.
But then $Q(p(\cdot),\lambda_1,\lambda_2)[q'(\cdot),q'(\cdot)]=\alpha
Q(p'(\cdot),\lambda'_1,\lambda'_2)[q'(\cdot),q'(\cdot)]\ge0$, contradicting the assumption.
\end{proof}

We now give two examples. The first one illustrates Corollary~\ref{S3}, while the
second one pertains to Theorem~\ref{T2} and shows that the hypotheses of the theorem
cannot be discarded.

{\bf Example 1}. Consider the dynamical system
\begin{multline}\label{ex1}
\dot x_1=u,\quad\dot x_2=u^2-x_1^2,\quad u(t)\in\mathbb R\, \,\,\,\text{for almost
all}\,\,\, t\in[0,T],\\ x_1(0)=x_2(0)= x_1(T)=x_2(T)=0,
\end{multline}
where $T>0$.

A process $(\widehat x (\cdot),\widehat u (\cdot))=(0,0)$, where $\widehat x (\cdot)=(\widehat x _1(\cdot),\widehat x _2(\cdot))$, is admissible for this
system. According to the general statement of the problem, here $f=0$, and so we assume that
$g=(x_1(0),x_2(0),x_1(T),x_2(T))^T$.  A~simple calculation shows that the pairs
$(p(\cdot),\lambda_2)=((0,\alpha),(0,\alpha,0,\alpha))$, where $\alpha\le0$, and only such
pairs satisfy the first two relations of~\eqref{lam}, and hence
$$
\Lambda_{\mathrm{max}}(0,0)=\{\,(p(\cdot),\lambda_2)=((0,\alpha),(0,\alpha,0,\alpha)),\,\,\alpha<0\,\}.
$$
Clearly, $\Lambda_{\mathrm{max}}(0,0)\cup\{0\}$ is a~pointed cone (ray), and hence
$(0,0)$ is not a~singular process, the order of abnormailty of system \eqref{ex1} with respect to
this process being equal to~1. We shall employ Corollary~\ref{S3}.

In our setting it is easily checked that
\begin{multline*}
K(0,0)=\{\,q(\cdot)=(h(\cdot),v(\cdot))\in AC([0,T],\mathbb R^2)\times L_\infty([0,T])\\
: \dot h_1(\cdot)=v(\cdot),\quad \dot h_2(\cdot)=0,\quad h_i(0)=h_i(T)=0,\,\,\, i=1,2\,\}.
\end{multline*}
Let $(p(\cdot),\lambda_2)\in \Lambda_{\mathrm{max}}(0,0)$. A~direct calculation shows that,
for any  $q(\cdot)\in K(0,0)$,
\begin{multline*}
Q(p(\cdot),\lambda_2)[q(\cdot),q(\cdot)]=-2\alpha\int_{0}^{T}(v^2(t)-h_1^2(t))\,dt=\\
-2\alpha\int_{0}^{T}(\dot h_1^2(t)-h_1^2(t))\,dt.
\end{multline*}
It is well known (and is easily checked) that the integral is nonnegative on $[0,T]$ if
$T\le\pi$; its values on  $[0,T]$ are negative if $T>\pi$. But then
$Q(p(\cdot),\lambda_2)[q(\cdot),q(\cdot)]$ with  $T>\pi$ assumes negative values, and hence
by Corollary~\ref{S3} our system is locally controllable with respect to the process $(0,0)$.

\vskip10pt

{\bf Example 2}. Now consider the dynamical system
\begin{multline}\label{ex2}
\dot x_1=u,\quad\dot x_2=u^3,\quad u(t)\in (a,+\infty)\, \,\,\text{for almost all}\,\,\,
t\in[0,1],\\  x_1(0)=x_2(0)=0,\quad x_1(1)=x_2(1)=1,
\end{multline}
where $a<1$.

The process $(\widehat x (\cdot),\widehat u (\cdot))$, where $\widehat x _1(t)=\widehat x _2(t)=t$, \ $\widehat u (t)=1$, \ $t\in[0,1]$,
is admissible for system \eqref{ex2}. A~direct analysis of the first two relations in~\eqref{lam}
shows that if a~pair $(p(\cdot),\lambda_2)$ satisfies these relations, then we necessarily
have $p(\cdot)=(\alpha,-\alpha/3)$, $\lambda_2=(\alpha,-\alpha/3,\alpha,-\alpha/3))$ and
\begin{equation}\label{9c}
\alpha\left(u-\frac{u^3}3\right)\le\alpha\,\frac23,\quad \forall\,\,u\in (a,+\infty),
\end{equation}
for some $\alpha\in\mathbb R$.

There are two cases to consider. $1)$ $a<-2$. In this setting, inequality \eqref{9c} is
possible only if $\alpha=0$. Hence $\Lambda_{\mathrm{max}}(\widehat x (\cdot),\widehat u (\cdot))=\emptyset$ and \textit{a~fortiori}
$\Lambda(\widehat x (\cdot),\widehat u (\cdot),q(\cdot))=\emptyset$ for any $q(\cdot)\in K(\widehat x (\cdot),\widehat u (\cdot))$, where
\begin{multline*}
K(\widehat x (\cdot),\widehat u (\cdot))=\{\,q(\cdot)=(h(\cdot),v(\cdot))\in AC([0,1],\mathbb R^2)\times L_\infty([0,1])\\
: \dot h_1(\cdot)=v(\cdot),\quad \dot h_2(\cdot)=3v(\cdot),\quad h_i(0)=h_i(1)=0,\,\,\, i=1,2\,\}.
\end{multline*}
Now Theorem~\ref{T2} shows that system \eqref{ex2} is locally controllable.

$2)$  $a\ge-2$. We claim that in this case
$\Lambda(\widehat x (\cdot),\widehat u (\cdot),q(\cdot))\ne\emptyset$ for any  $q(\cdot)\in K(\widehat x (\cdot),\widehat u (\cdot))$ and
at the same time system \eqref{ex2} is not locally controllable with respect to the process $(\widehat x (\cdot),\widehat u (\cdot))$ (that is,
the condition
$\Lambda(\widehat x (\cdot),\widehat u (\cdot),q(\cdot))=\emptyset$ is essential for local controllability).

Indeed, inequality \eqref{9c} holds for any $\alpha>0$ and becomes an
equality at the point $u=1$. Let $(p(\cdot),\lambda_2)\in \Lambda_{\mathrm{max}}(0,0)$.
It is easily checked that $Q(p(\cdot),\lambda_2)[q(\cdot),q(\cdot)]=2\alpha\int_0^1 v^2(t)\,dt$ for any
$q(\cdot)=(h(\cdot), v(\cdot))\in K(\widehat x (\cdot),\widehat u (\cdot))$ and therefore,
$\Lambda(\widehat x (\cdot),\widehat u (\cdot),q(\cdot))\ne\emptyset$.

To show that the system is not controllable with respect to the process $(\widehat x (\cdot),\widehat u (\cdot))$
it clearly suffices to show that, for any $\varepsilon>0$, there is no process
$(x_y(\cdot), u_y(\cdot))$ ($x_y(\cdot)=(x_{1y}(\cdot),x_{2y}(\cdot))$) that satisfies
the differential equation in~\eqref{ex2}, $u_y(t)\in(a,+\infty)$ for almost all
$t\in[t_0,t_1]$, $(x_{1y}(0),x_{2y}(0))=(0,0)$ and $(x_{1y}(1),x_{2y}(1))=(1,
1-\varepsilon)$.
Indeed, if this it were so, then setting $\eta_y(\cdot)=u_y(\cdot)-1$ and taking into account that $\eta_y(t)>-3$ for almost all
$t\in[t_0,t_1]$ and $\int_0^1\eta_y(t)\,dt=0$, we arrive at the contradiction:
\begin{multline*}
-\varepsilon=x_{2y}(1)-1=\int_0^1u_y^3(t)\,dt-1=\int_0^1(1+\eta_y(t))^3\,dt-1
\\=3\int_0^1\eta_y(t)\,dt+\int_0^1\eta_y^2(t)(3+\eta_y(t))\,dt\ge0.
\end{multline*}
Thus the system \eqref{ex2} is not locally controllable with respect to the process $(\widehat x (\cdot),\widehat u (\cdot))$.

\end{document}